\documentclass{article}
\usepackage{cite}
\usepackage{graphicx}
\usepackage{amsmath,amssymb,latexsym}
\usepackage{amsthm}
\usepackage{hyperref}
\usepackage{fullpage}
\usepackage{color}
\usepackage{float}

\newcommand{\set}[1]{\left\{ #1\right\}}
\newcommand{\abs}[1]{\left| #1\right|}

\newcommand{\bR}{\displaystyle \mathbb{R}}

\newcommand{\bF}{\displaystyle \mathbb{F}}
\newcommand{\bP}{\displaystyle \mathbb{P}}
\newcommand{\bN}{\displaystyle \mathbb{N}}
\newcommand{\bZ}{\displaystyle \mathbb{Z}}

\newtheorem{theorem}{Theorem}[section]
\newtheorem{lemma}[theorem]{Lemma}
\newtheorem{prop}[theorem]{Proposition}
\newtheorem{remark}[theorem]{Remark}

\newtheorem{cor}[theorem]{Corollary}
\newtheorem{conj}[theorem]{Conjecture}
\newtheorem*{conj*}{Conjecture}
\newtheorem{que}[theorem]{Question}
\newtheorem{observation}[theorem]{Observation}

\begin{document}
	\title{	Indistinguishable sceneries on the Boolean hypercube}
	\author{Renan Gross\footnote{Weizmann Institute of Science. renan.gross@weizmann.ac.il}
		\;and Uri Grupel\footnote{Weizmann Institute of Science. uri.grupel@weizmann.ac.il. Supported by the European Research Council (ERC).}}
	\date{}
	\maketitle
	
	\begin{abstract}
		We show that the scenery reconstruction problem on the Boolean hypercube is in general impossible. This is done by using locally biased functions, in which every vertex has a constant fraction of neighbors colored by $1$, and locally stable functions, in which every vertex has a constant fraction of neighbors colored by its own color. Our methods are constructive, and also give super-polynomial lower bounds on the number of locally biased and locally stable functions. We further show similar results for $\bZ^n$ and other graphs, and offer several follow-up questions.
	\end{abstract}
	

	\section{Introduction}
	Let $f : \set{-1,1}^{n} \rightarrow \set{-1,1}$ be a Boolean function on the $n$-dimensional hypercube, and let $S_i$ be a random walk on the hypercube. Can we reconstruct the function $f$ (with probability 1, up to the hypercube's symmetries) by only observing the scenery process $\set{f(S_i)}_{i}$?
	
	Similar questions have been raised for other graphs. For example, it was shown in~\cite{benjamini+kesten96} that when $G$ is a cycle graph, the answer is yes: it is possible to reconstruct the function $f$ (which is a string up to choice of origin) up to rotation and reflection with probability $1$. It is still an open question whether any such string can be reconstructed in polynomial time. When $G=\bZ$, reconstruction is generally impossible \cite{lindenstrauss99}; for random sceneries on $\bZ$ see~\cite{matzinger_rolles03}.
	
	When $G$ is the hypercube, such a process was studied for a specific Boolean function, the {\em percolation} crossing, under the notion of dynamical percolation; see~\cite{garban+steif15} for details. 
	
	In the general case, however, we show that for $n\geq 4$ the answer is no. We do this by considering a pair of non-isomorphic functions $f$ and $g$ such that if $S_i$ and $T_i$ are random walks on the hypercube, then $f(S_i)$ and $g(T_i)$ have exactly the same distribution. We discuss two different classes of such functions:
	
	\begin{itemize}
		\item{\textbf{Locally $p$-biased functions:}} Let $G$ be a graph. A Boolean function $f: G \rightarrow \set{-1,1}$ is called \emph{locally $p$-biased}, if for every vertex $x\in G$ we have
		$$\frac{\abs{\{y\sim x;\;f(y) = 1\}}}{deg(x)}=p.$$
		In words, $f$ is locally $p$-biased if for every vertex $x$, $f$ takes the value $1$ on exactly a $p$-fraction of $x$'s neighbors. If $f$ is a locally $p$-biased function, then the random variables $\{f(S_i)\}_i$ have the same distribution as independent Bernoulli random variables with $\bP(f(S_i)=1)=p$. 
			
		\item{\textbf{Locally $p$-stable functions:}} Let $G$ be a graph. A Boolean function $f: G \rightarrow \set{-1,1}$ is called \emph{locally $p$-stable}, if for every vertex $x\in G$ we have
		$$\frac{\abs{\{y\sim x;\;f(x) = f(y)\}}}{deg(x)}=p.$$
		In words, $f$ is locally $p$-stable if for every vertex $x$, $f$ retains its value on exactly a $p$-fraction of $x$'s neighbors. If $f$ is locally $p$-stable, then the random variables $\{f(S_i)f(S_{i+1})\}_{i}$ have the same distribution as independent Bernoulli random variables with $\bP(f(S_i)f(S_{i+1})=1)=p$.
	\end{itemize}
	
	We say that two Boolean functions $f,g:\{-1,1\}^n\to\{-1,1\}$ are \emph{isomorphic}, if there exists an automorphism of the hypercube $\psi:\{-1,1\}^n\to\{-1,1\}^n$ such that $f\circ\psi = g$. Two functions are \emph{non-isomorphic} if no such $\psi$ exists.
		
	The existence of two non-isomorphic locally $p$-biased functions, or two non-isomorphic locally $p$-stable functions thus render scenery reconstruction on the hypercube impossible. 
	
	It is not immediately obvious that pairs of non-isomorphic locally $p$-biased and pairs of non-isomorphic locally $p$-stable functions exist. It is then natural to ask, for which $p$ values do they exist? If they do exist, how many of them are there?
		
	In this paper, we characterize the possible $p$ values on the $n$-dimensional hypercube, give bounds on the number of non-isomorphic pairs, and discuss results on other graphs. The paper is organized as follows.
	
	In \S~\ref{sec:existence} we give a full characterization of the connection between the dimension of the hypercube $n$ and the permissible $p$ values of locally $p$-biased functions, as expressed in the following theorem:
	
	\begin{theorem}\label{thm:main_theorem}
		Let $n \in \bN$ be a natural number and $p \in [0,1]$. There exists a locally $p$-biased function $f:\set{-1,1}^{n} \rightarrow \set{-1,1}$ if and only if $p = b/2^{k}$ for some integers $b \geq 0, k \geq 0$, and $2^{k}$ divides $n$. 
	\end{theorem}	
	
	Our proof can construct functions for all $p$ of the above form.
	
	In \S~\ref{sec:non_unique} we inspect the class size of non-isomorphic locally $p$-biased functions on the hypercube. We show that the class size for $p=1/2$ is at least $C2^{\sqrt{n}}/n^{1/4}$ for some constant $C > 0$, and for $p=1/n$ is super-exponential in $n$, when such $p$ values are permissible. Thus reconstruction is impossible for such functions. We conjecture that the number of non-isomorphic locally $p$-biased functions scales quickly for all permissible $p$ values:
	\begin{conj}
		Let $n>0$ be even. Let $p=b/2^k$, where $1\leq b\leq 2^k$, $k\geq 1$ and $2^k$ divides $n$. Let $B_p^n$ be the set of non-isomorphic locally $p$-biased functions. Then $\abs{B_p^n}$ is super-exponential in $n$.
	\end{conj}	
	
	In \S~\ref{sec:locally_stable_functions} we briefly discuss locally $p$-stable functions. We show that they exist for all possible $p$ values, and that for most $p$ values there are many non-isomorphic pairs; however, for every $n$, there are $p$ values for which there is a single unique locally $p$-stable function. The results in this section are based on those of \S~\ref{sec:non_unique}.
	
	In \S~\ref{sec:open_questions} we discuss other graphs. First, we show that when $G$ is a regular tree of degree $n$, then all $p=a/n$ are permissible. Second, we show that for $G=\bZ^n$ all the results for the hypercubes hold true. This gives us a partial answer for permissible $p$ values for $\bZ^n$, but there are additional values that cannot be achieved through the hypercube construction: for example, for $n=1$ we can define a function with $p=1/2$ and when $n=2$ we can find a function with $p=1/4$. We also discuss other Cayley graphs of $\bZ$, and suggest further questions on scenerey reconstruction.
	
	Throughout most of this paper we treat the Boolean hypercube as the set $\set{-1,1}^n$. We identify it with the $\set{0,1}^n$ hypercube by considering $-1$ in the first to correspond to $0$ in the second.		
	
	
	\section{Characterization of permissible $p$ values}\label{sec:existence}
	In this section we prove Theorem \ref{thm:main_theorem}. The ``only if'' part is achieved by a double counting argument.
	\begin{proof}[Proof (of the ``only if'' statement of Theorem \ref{thm:main_theorem})]
		Suppose that $f:\set{-1,1}^{n} \rightarrow \set{-1,1}$ is a locally $p$-biased Boolean function. Let $x$ be a uniformly random element of $\set{-1,1}^{n}$. Then $f(x)$ is $f$'s value on a uniformly random point of the hypercube, and is equal to $1$ with probability $l/2^{n}$, where $l = \abs{\set{x \in \set{-1,1}^n ;\; f(x) = 1}}$ is the number of vertices on which $f$ obtains the value $1$. Now let $y$ be a uniformly random neighbor of $x$. The function $f$ is locally $p$-biased, so probability that $f(y) = 1$ is $p$ by definition. Since both $x$ and $y$ are uniform random vertices, $\bP(f(x) = 1) = \bP(f(y) = 1)$. Denoting $p = m/n$ for some $m \in \set{0,1,\ldots,n}$, this gives
		\begin{equation}
		p = \frac{l}{2^{n}} = \frac{m}{n}. \label{eq:double_counting}
		\end{equation}
		Decompose $n$ into its prime powers, writing $n = c2^{k}$, where $c$ is odd. Then by (\ref{eq:double_counting}), we have that
		
		$$
		l = \frac{2^{n-k} \cdot m}{c}
		$$
		is an integer, and so $c$ must divide $m$, i.e $m = bc$ for some $b$. But then
		
		$$ p = \frac{m}{n} = \frac{b}{2^{k}} $$
		as stated by the theorem.
	\end{proof}
	
	The ``if'' part of Theorem \ref{thm:main_theorem} is given by an explicit construction, performed in three steps. First, we use perfect codes in order to obtain a locally $1/n$-biased function for $n$ that is a power of two. Second, we extend the result to a locally $m/n$-biased function by taking the union of $m$ locally $1/n$-biased functions with disjoint support. Finally, given a locally $p$-biased function on $n$ bits, we show how to manipulate its Fourier representation in order to yield a locally $p$-biased function on $cn$ bits for any $c$.

	We begin with a brief review of binary codes. We omit proofs and simply state definitions and known results; for a more thorough introduction, see e.g \cite{lint1975, lint1998}.
	
	A binary code $C$ on the $n$-dimensional hypercube is simply a subset of $\set{-1,1}^n$; its elements are called \emph{codewords}. The \emph{distance} of a code $C$ is defined as $\min_{x,y \in C} \delta_H(x,y)$, where $\delta_H(x,y)  = \abs{\set{i \in \set{1,\ldots,n}; x_i \neq y_i}}$ is the Hamming distance between $x$ and $y$, that is, the number of coordinates in which $x$ and $y$ differ. A code of odd distance $d$ is called \emph{perfect} if the Hamming balls of radius $(d-1)/2$ around each codeword completely tile the hypercube without overlaps. A code is called \emph{linear} if its codewords form a vector space over $\bF_2$.
	
	A particularly interesting code is the Hamming code with $k$ data bits, denoted $H_k$. It is a linear, distance-$3$ perfect code on the hypercube of dimension $n = 2^k -1$. Its codewords are structured as follows. For $x \in H_k$ and $i \in \set{1,\ldots,n}$, the bit $x_i$ is called a \emph{parity bit} if $i$ is a power of $2$, and \emph{data bit} otherwise. Thus every codeword contains $k$ parity bits and $2^k - k -1$ data bits. The data bits range over all 
	possible bit-strings on $2^k -k -1$ bits, while the parity bits are a function of the data bits: 	
	$$x_i = \bigoplus_{j: i \land j \neq 0} x_j \qquad  \forall i = 2^l, l\geq0$$ 
	where $\oplus$ denotes exclusive bitwise or (xor), and $\land$ denotes bitwise AND. In words, the parity bit $x_i$ is equal to the xor of all data bits $x_j$ such that the bitwise AND between $i$ and $j$ is non-zero. Thus there are $2^k - k - 1$ codewords in $H_k$. 
		
	Armed with perfect codes, we are ready to start our proof.
	
	\begin{lemma}\label{lemma:1_over_n_power_of_2}
		Let $n = 2^{k}$ be a power of two. Then there exists a locally $1/n$-biased function on $\set{-1,1}^{n}$.
	\end{lemma}
	
	\begin{proof}
		In a locally $1/n$-biased function $f$, every point in the hypercube must have exactly 1 neighbor which is given the value 1, and $n-1$ neighbors which are given the value -1.
		
		Let $C$ be a distance-$3$ perfect code on the $n-1 = 2^{k} - 1$-dimensional hypercube. That is, every two code words in $C$ are at a Hamming distance of at least $3$ from each other, and the Hamming balls of radius $1$ centered around each codeword completely tile the hypercube. Such codes exist for dimension $2^{k}-1$; for example, as mentioned above and shown in \cite{lint1975}, the Hamming code is such a code. Define $f:\set{-1,1}^n \rightarrow \set{-1,1}$ to be the following function:
		
		$$ f(x) = \begin{cases}
		1, & x\in C\times\{-1,1\} \\
		-1, & \text{otherwise.}\\
		\end{cases} $$
		In words, $f(x)$ takes the value of $1$ whenever the first $n-1$ coordinates of $x$ are a codeword in $C$, and otherwise takes the value of $-1$. Then $f$ is a locally $1/n$-biased function:
		\begin{list}{-}{}
			\item 
			If $f(x) = 1$, then $x=(y,b)\in C\times\{-1,1\}$. Thus $x' = (y,-b)$ is the only neighbor of $x$ on which $f(x') = 1$; any other neighbor differs from $x$ in the first $n-1$ coordinates, and since $C$ is a distance-$3$ code, these coordinates are not a codeword in $C$.
			\item 
			If  $f(x) = -1$, then $x=(y,b)$ where $b\in\{-1,1\}$ and $y$ is not a codeword of $C$.
			Since $C$ is perfect, $y$ must fall inside some radius-$1$ ball of a codeword $z$. Then $x' = (z,b)$ is the only neighbor of $x$ such that $f(x') = 1$; any other codeword differs from $z$ in at least $3$ coordinates since $C$ is a distance-$3$ code, and so differs from $y$ in at least $2$.
		\end{list}
	\end{proof}
	
	\begin{lemma}\label{lemma:m_over_n_power_of_2}
		Let $n = 2^{k}$ be a power of two. Then there exists a locally $m/n$-biased function on $\set{-1,1}^{n}$ for any $m = 0, 1, \ldots, n.$
	\end{lemma}
	
	\begin{proof}
		For $m = 0$ the statement is trivial. Let $m \in \set{1,\ldots,n}$. In order to construct a locally $m/n$-biased function, it is enough to find $m$ locally $1/n$-biased functions $f_{1}, \ldots, f_{m}$ with pairwise disjoint support, i.e $\set{x : f_{i}(x) = 1} \cap \set{x : f_{j}(x) = 1} = \emptyset$ for all $i \neq j$. With these functions, we can define $f$ in the following manner:
		
		$$ f(x) = \begin{cases}
		1, & f_i(x) = 1 \text{ for some $i$}\\
		-1, & \text{otherwise.}\\
		\end{cases} $$
		Then $f$ is a locally $m/n$-biased function: For every $x \in \set{1,-1}^n$, consider its neighbors on which $f$ takes the value $1$, i.e the set $\set{y ;\; d(x,y) = 1, \exists i \ s.t \ f_{i}(y) = 1}$. Each $f_{i}$ contributes exactly one element to this set, since it is a locally	$1/n$-biased function; further, these elements are all distinct, since the $f_{i}$'s have pairwise disjoint supports. So $x$ has $m$ neighbors on which $f$ takes the value 1.
		
		Recall that the Hamming code on $2^{k}-1$ bits uses $2^{k}-k-1$ data bits (these range over all possible bit-strings on $2^{k}-k-1$ bits) and $k$ parity bits (these are a function of the data bits). Let $C$ be the Hamming code on $2^{k}-1$ bits, and rearrange the order of the bits so that the parity bits are all on the right hand side of the codeword, i.e each codeword $x$ can be written as $x = (y,z)$ where $y$ is a word of length $2^{k}-k-1$ constituting the data bits and $z$ is a word of length $k$ constituting the parity bits.
		
		Now, for all $1 \leq i \leq n$, define the sets $C_{i} = \set{x \oplus (i-1) ;\; x \in C}$, where $\oplus$ denotes the exclusive or (xor) operator. Then the sets $C_i$ are all pairwise disjoint: in order for two words $x = (y,z) \in C_{i}$ and $x' = (y',z') \in C_{j}$ to be the same,  we need to have both $y = y'$ and $z = z'$. But if $y = y'$ then the data bits are the same, and by construction $z \oplus z' = (i-1) \oplus (j-1)$, so $z \neq z'$ if $i \neq j$. Further, since xoring by a constant only amounts to a rotation of the hypercube, each $C_{i}$ is still a perfect code.
		
		Let $f_{i}$ be the function which uses $C_{i}$ as its perfect code in the proof of Lemma~\ref{lemma:1_over_n_power_of_2}. Then, $f_{1}, \ldots, f_{n}$ are $n$ locally $1/n$-biased functions  with pairwise disjoint supports. The combination of any $m$ of these functions yields a locally $m/n$-biased function.
		
	\end{proof}

	\begin{lemma}\label{lemma:simulating_p}
		Let $f:\set{-1,1}^{n} \rightarrow \set{-1,1}$ be a locally $p$-biased function on the $n$-dimensional hypercube.
		Let $c \in \bN$, and define a new function $f':\set{-1,1}^{cn} \rightarrow \set{-1,1}$ by
		\begin{eqnarray}
		f'(x) & = &  f\left(\prod_{j=0}^{c-1}x_{1+jn},\ldots,\prod_{j=0}^{c-1}x_{n+jn}\right).
		\end{eqnarray}
		Then $f'$ is a locally $p$-biased function. 
		
	\end{lemma}
	
	\begin{proof}
		Let $x' \in \set{-1,1}^{cn}$ be a point on the $cn$-dimensional hypercube, and let $y \in \set{-1,1}^{n}$ be such that $y_{i} = x'_{i} \cdot \ldots \cdot x'_{(c-1)n+i}$. Then by definition, $ f'(x') = f(y)$. Since $f$ is a locally $p$-biased function, $y$ has $pn$ neighbors on which $f$ takes the value $1$. Each of these neighbors is obtained from $y$ by flipping a single coordinate $y_{i}$; this amounts to changing any one of the $c$ coordinates of $x'$ which make up the $y_{i}$. Since the $y_{i}$'s are disjoint monomials in the coordinates of $x'$, this implies that there are at least $pcn$ neighbors of $x'$ on which $f'$ takes the value $1$. 
		
		The same argument can be repeated for the value $-1$ instead of $1$, showing that there are at least $(1-p)cn$ neighbors of $x'$ on which $f'$ takes the value $-1$. But since the number of neighbors of $x'$ is $nc$, the inequalities are in fact equalities. Hence, there are exactly $pcn$ neighbors of $x'$ on which $f'$ takes the value $1$, completing the proof.
	\end{proof}
	
	We are now ready to prove that the condition on $p$ is sufficient in Theorem~\ref{thm:main_theorem}.
	
	\begin{proof}[Proof (of the ``if'' statement of Theorem \ref{thm:main_theorem})]
		All that is left is to stitch the above lemmas together: Let $n = c2^{k}$. Using Lemma \ref{lemma:m_over_n_power_of_2}, create a locally $p$-biased function $g:\set{-1,1}^{2^k} \rightarrow \set{-1,1}$ on $2^{k}$ variables; then, using Lemma \ref{lemma:simulating_p}, extend it to a function $f$ on $n$ variables.
	\end{proof}

	
	\section{Non-isomorphic functions}\label{sec:non_unique} 
	In this section we discuss the classes of non-isomorphic locally $p$-biased function.
	We show that for the hypercube of dimension $n$, the growth rate with respect to $n$ is at least $\Omega\left(2^{\sqrt{n}}/\sqrt{n}\right)$ for $p=1/2$ and super-exponential for $p=1/n$, when such $p$'s are permissible. We conjecture that for any permissible $p$ the growth rate is super-polynomial.	
	
	The proof for $p=1/2$ is based on an explicit construction of non-isomorphic locally $1/2$-biased functions.
	In order to define these functions we use the following simple proposition.
	\begin{prop}\label{prop:product}
		Let $f_i:\{-1,1\}^{n_i}\to \{-1,1\}$ be locally $1/2$-biased functions for $i=1,2$ where $n_1+n_2=n$. Then
		$$f(x)=f_1(x_1,...,x_{n_1})f_2(x_{n_1+1},...,x_n)$$
		is a locally $1/2$-biased function on $\{-1,1\}^n$.
	\end{prop}
	
	\begin{proof}
		Let $x\in \{-1,1\}^n$, let $x'$ be a neighbor of $x$, and denote $f(x)=f_1(x_1,\ldots,x_{n_1})f_2(x_{n_1+1},\ldots,x_n)=y_{1}\cdot y_{2}$ and $f(x')=f_1(x'_{1},\ldots,x'_{n_1})f_2(x'_{n_1+1},\ldots,x'_{n})=y'_{1}\cdot y'_{2}$ . Then $x'$ differs from $x$ in either the first $n_{1}$ coordinates, or the last $n_{2}$ coordinates. If it differs in the first $n_{1}$ coordinates, then $y'_{2} = y_{2}$. Since $f_{1}$ is locally $1/2$-biased, there are exactly $n_{1}/2$ coordinate changes such that $y'_{1} = y'_{2}$, yielding $f(x') = 1$. Similarly, if $x'$ differs in the last $n_{2}$ coordinates, then $y'_{1} = y_{1}$, and there are exactly $n_{2}/2$ coordinate changes such that $y'_{2} = y'_{1}$, again yielding $f(x') = 1$. So overall, $x$ has exactly $n_1/2 + n_2/2 = n/2$ neighbors where $f$ is $1$.
	\end{proof}
	
	The above proposition allows us to construct examples for locally $1/2$-biased functions, by combinations of such functions on lower dimensions.\\
	
	We have two basic examples for locally $1/2$-biased functions:
	\begin{enumerate}
		\item
		In any even dimension $n$,
		$$g_n(x_1,\ldots,x_n)=x_1\cdots x_{n/2}.$$
		\item \label{eq:basic_function}
		In dimension $n=4$,
		$$h(x_1,x_2,x_3,x_4)=\frac{1}{2}\left(x_1x_2+x_2x_3-x_3x_4+x_1x_4\right).$$
	\end{enumerate}
	
	The Fourier decomposition of a Boolean function is its expansion as a real multilinear polynomial: any Boolean function $f:\set{-1,1}^n \rightarrow \bR$ can be written as a sum
	$$ f(x_1,\ldots,x_n) = \sum_{S\subseteq \set{1,\ldots,n}} \hat{f}_S \prod _{i\in S} x_i, $$
	where the $\hat{f}_S$ are real coefficients. Such a representation is unique; for a proof and other properties of the Fourier decomposition, see e.g chapter 1 in \cite{oDonnell2014}. 
	
	Automorphisms of the hypercube are manifested on the Fourier decomposition of a Boolean function by either permutation or by a sign change to a subset of indices. Hence, we can show that two Boolean functions are not isomorphic by showing that their Fourier decompositions cannot be mapped into one another by such permutations and sign changes.
	
	In this section, a tensor product of two functions $f(x_1, \ldots, x_n)$ and $g(x_1,\ldots, x_m)$ is a function on disjoint indices, i.e. 
	$$ h(x_1, \ldots, x_{n+m}) = f(x_1,\ldots,x_n)\cdot g(x_{n+1},\ldots, x_{n+m}).$$
	
	\begin{prop} \label{prop:half_biased}
		There exists $h_1,h_2,\ldots$ such that for any $k$ the function $h_k$ is locally $1/2$-biased on the $4k$-dimensional hypercube and $h_k$ is not isomorphic to any tensor product of $h_1,\ldots,h_{k-1},g_2,g_4,g_6,...$.
	\end{prop}
	
	\begin{proof}
		We define $h_1=h$, and
		$$h_k=h\left(\prod_{i=0}^{k-1} x_{1+4i},\ldots,\prod_{i=0}^{k-1} x_{4+4i}\right),$$
		where $h$ is the function from example~\ref{eq:basic_function}.
		By Lemma \ref{lemma:simulating_p}, $h_k$ is locally $1/2$-biased on $\{-1,1\}^{4k}$.
		Assume that $h_k$ is isomorphic to a tensor product of $h_1,\ldots,h_{k-1},g_2,...,g_{n-2}$, as in Proposition \ref{prop:product}. If there exists $1\leq i \leq j< k$ such that both $h_i$ and $h_j$ appear in a product that is isomorphic to $h_k$, then the Fourier decomposition of the product would have at least $16$ different monomials. But $h_k$ has only $4$ different monomials, and the functions cannot be isomorphic. 
		Similarly, if we do not use any of the functions $h_1,\ldots,h_{k-1}$, then we get the parity function, which has only one monomial in its Fourier decomposition.
		Hence, we may assume that there is only one $1 \leq i<k$ such that $h_i$ is in the product. Then, up to an automorphism, this function is of the form 
		$$f(x) = h_i(x_1,\ldots,x_{4i})g_{4k-4i}(x_{4i+1},\ldots,x_{4k}).$$
		On the one hand, by definition of $h_k$, its Fourier decomposition has pairs of monomials with no shared indices (e.g. the monomials that replace $x_1x_2$ and $x_3x_4$ in $h_{1}$).
		On the other hand, in the decomposition of $f$, all monomials have shared indices; for example $x_{4i+1}$ appears in all monomials. Hence they are not isomorphic.
	\end{proof}
	
	Using the functions $h_1,h_2,\ldots$ we can give a lower bound for the class of non-isomorphic locally $1/2$-biased functions.
	
	\begin{lemma} \label{lemma:num_of_solutions}
		The number of non-negative integer solutions to
		\begin{equation} \label{eq:integer_inequality}
		a_1+2a_2+\ldots+ka_k \leq k
		\end{equation}
		is at least $C4^{\sqrt{k}}/k^{1/4}$, where $C>0$ is a universal constant.
	\end{lemma}
	
	\begin{proof}
		For any $1\leq \ell\leq k$, the number of solutions to (\ref{eq:integer_inequality}) is at least the number of solutions to
		$$\ell a_1+\ell a_2+\ldots+\ell a_\ell \leq k.$$
		It is well known that the number of solutions to this inequality is
		$$\binom{\ell+k/\ell}{\ell}.$$
		This term is maximized when $\ell^2 = k$. Hence, a lower bound for the number of solutions to (\ref{eq:integer_inequality}) is
		$$\binom{2\sqrt{k}}{\sqrt{k}}.$$
		By Stirling's formula, the asymptotic of this is $(1/\sqrt{\pi})4^{\sqrt{k}}/k^{1/4}$.
	\end{proof}
	
	\begin{remark}
		The number of integer solutions to the equality case is the famous partition function $p(n)$.
		Hardy and Ramanujan \cite{hardy+ramanujan18} showed precise asymptotics. 
		Using their result it is possible to show that the number of integer solutions is
		$$\sum_{j=1}^k p(j) \sim Ce^{c\sqrt{k}}/\sqrt{k},$$
		with explicit constants $C,c>0$. For our purposes, the simple estimation in Lemma \ref{lemma:num_of_solutions} is enough.
	\end{remark}
	
	\begin{prop} \label{prop:number_of_half_biased}
		Let $n$ be even. Let $B_{1/2}^n$ be a maximal class of non-isomorphic locally $1/2$-biased functions, i.e every two functions in $B_{1/2}^n$ are non-isomorphic to each other. Then $\abs{B_{1/2}^n} \geq C2^{\sqrt{n}}/n^{1/4}$, where $C>0$ is a universal constant.
	\end{prop}
	
	\begin{proof}
		Let $k=\lfloor n/4\rfloor$. By Proposition \ref{prop:product}, we can construct locally $1/2$-biased functions by tensor products of $h_1,\ldots,h_k$ and $g_1,\ldots,g_n$, as follows: choose $j$ functions $\set{h_{i_j}}$ such that $m := \sum 4 i_j \leq n$. Then the tensor product $\otimes h_{i_j}$ uses $m$ variables. This can be completed to $n$ variables by tensoring with $g_{n-m}$.
						
		If two functions use the same $h_i$'s, then they are isomorphic (by change of indices). And if they have a different decomposition of $h_i$'s, then by the same arguments used in Proposition \ref{prop:half_biased}, they have a different Fourier decomposition and are therefore non-isomorphic. Thus, the isomorphic class of such a function is determined by the number of times each $h_i$ appears in the product. 
		
		Hence, the number of non-isomorphic functions we can construct in this manner is the number of solutions to
		\begin{equation} \label{eq:combinations_inequality}
		4a_1+8a_2+\cdots+4ka_k \leq n
		\end{equation}
		where the $a_1,\ldots,a_k$ are non-negative integers that represent the number of copies of $h_i$ in the product. Using Lemma \ref{lemma:num_of_solutions}, this number is at least $C4^{\sqrt{k}}/k^{1/4}=C'2^{\sqrt{n}}/n^{1/4}$
	\end{proof}
	
	It should be noted that a locally $1/2$-biased function has a natural condition on its Fourier decomposition.
	It might be possible to obtain better bounds on the number of non-isomorphic functions using this condition.
	
	\begin{prop}
		Let $f:\{-1,1\}^n\to\{-1,1\}$ be a locally $1/2$-biased function. Then the Fourier weight at degree $n/2$ is $1$.
	\end{prop}
	
	\begin{proof}
		Let $A_n$ be the adjacency matrix of the hypercube. The map
		$$f\to (f(a_1),...,f(a_{2^n})),$$
		where $a_1,...,a_{2^n}$ are the vertices of the hypercube, is a bijection between locally $1/2$-biased functions and the null space of $A_n$.	Since
		$$A_n=\begin{pmatrix}
		A_{n-1} & I\\
		I & A_{n-1}
		\end{pmatrix},$$
		We have
		$$P_n(t)=P_{n-1}(t-1)P_{n-1}(t+1),$$
		where $P_n$ is the characteristic polynomial of $A_n$. 
		For $A_2$ the eigenvalue $0$ has multiplicity $2$ and $\pm 2$ has $1$. Continuing by induction, the eigenvalues of $A_m$ are $-m, -m+2,...,m$ with multiplicities $\binom{m}{0}, \binom{m}{2},..., \binom{m}{m}$.	
		Hence, for even $n$ the dimension of the null space is $\binom{n}{n/2}$.
		For any $S\subseteq\{1,2,...,n\}$ with $\abs{S}=n/2$ we denote $\chi_S(x)=\prod_{i\in S}x_i$.
		These functions are all locally $1/2$-biased, hence we can define the image in the null space by the bijection with $v_S$. Note that there are $\binom{n}{n/2}$ such vectors, and they form an independent set. Hence the set $\{v_S\}_S$ is a basis of the null set. By the bijection we get that every locally biased function is a linear combination of $\chi_S$.
	\end{proof}
	
	Class sizes for locally $1/n$-biased functions can also be achieved via the following proposition.
	
	\begin{prop} \label{prop:non_isomorphic_1_n_functions}
		Let $n=2^{k}$, and let $C_{1}$ and $C_{2}$ be two non-isomorphic distance-$3$ perfect codes on the $n-1$-dimensional hypercube. Then the two functions $f_{1}$ and $f_{2}$ obtained by using $C_{1}$ and $C_{2}$ as the perfect codes in the proof of Lemma \ref{lemma:1_over_n_power_of_2} are non-isomorphic. 
	\end{prop}
	
	\begin{proof}
		Suppose to the contrary that $f_{1}$ and $f_{2}$ are isomorphic, i.e there is an automorphism $\varphi: \set{-1,1}^{n} \rightarrow \set{-1,1}^n$ such that for all $x \in \set{-1,1}^{n}$, we have $f_{1}(x) = f_{2}(\varphi(x))$. Denote by $B =\set{(y,1) ;\; y\in \set{-1,1}^{n-1}}$ the $n-1$-dimensional hypercube obtained by fixing the last coordinate to $1$, denote $C = \set{(y,1) ;\; y \in C_{2}}$ and note that $\text{support}(f_{2}|_{B}) = C$ by construction. Consider  $\varphi|_{B}$, the restriction of $\varphi$ to $B$. This restriction is an isomorphism between $B$ and some $n-1$-dimensional hypercube $A$ contained within the $n$- dimensional hypercube. Any sub-hypercube of dimension $n-1$ is obtained from $\set{-1,1}^{n}$ by fixing one of the coordinates to be either $1$ or $-1$, and taking the span of all other coordinates. Then $A$ must be spanned by the first $n-1$ coordinates, leaving the last coordinate fixed: otherwise, by construction of Lemma \ref{lemma:1_over_n_power_of_2}, the set $A$ would contain two neighboring points $x$ and $x'$ that differ only in their last coordinate such that $f_{1}(x) = f_{1}(x') = 1$. This means there are $y,y' \in C$ obeying $\varphi(y) = x, \varphi(y') = x'$; but this is a contradiction, since $\varphi$ should preserve distances, and the distance between $x$ and $x'$ is $1$ while the distance between $y$ and $y'$ is $3$. So $A = \set{(y,b) ;\; y \in \set{-1,1}^{n-1}}$ for some $b \in \set{-1,1}$. But then $\varphi|_{B}$ is an isomorphism between $C_{1}$ and $C_{2}$, since $C_{1} $ is a perfect code in $A$ and $C_{2}$ is a perfect code in $B$; a contradiction. 
	\end{proof}
	
	\begin{cor}
		Let $n=2^{k}$. Let $B_{1/n}^n$ be the class of non-isomorphic locally $1/n$-biased functions. Then $\abs{B_{1/n}^n}$ is super-exponential in $n$.
	\end{cor}
	
	\begin{proof}
		By Proposition \ref{prop:non_isomorphic_1_n_functions}, any lower bound on the number of non-isomorphic perfect codes on the $n-1$-dimensional hypercube gives a lower bound to the number of locally $1/n$-biased functions on the $n$-dimensional hypercube. Recent constructions, such as in \cite{Krotov2008}, give a super-exponential lower bound on the number of such perfect codes.
	\end{proof}
	
	We would have liked to apply the same argument to locally $m/n$-biased functions, as given by the construction in Lemma \ref{lemma:m_over_n_power_of_2}. Our argument there used the explicit construction of the Hamming code which, being linear, was easy to modify in order to obtain functions with disjoint supports. Such is not the case for the construction of non-linear codes. However, we still believe that similar estimates are true for any permissible $p$.
	
	\begin{cor}
		By Proposition \ref{prop:number_of_half_biased}, scenery reconstruction is impossible for even-dimensional hypercubes. 
	\end{cor}
	For odd dimensional hypercubes, on which there are no non-trivial locally biased functions, we use locally stable functions instead, as described in the next section.
	
	
	\section{Locally $p$-stable functions}\label{sec:locally_stable_functions}
	Unlike locally $p$-biased functions, there is no restriction on permissible $p$ values for locally $p$-stable functions:
	\begin{observation}
		Let $p = m/n$ for some $m \in \set{0,1,\ldots,n}$. Then the parity function on $n-m$ variables, 
		$$ f(x_1,\ldots, x_n) = x_{m+1} x_{m+2}\ldots x_n $$
		is locally $p$-stable.
	\end{observation}

	Thus we will focus on the number of non-isomorphic pairs of locally stable functions. A negative result is attainable by a simple examination: 
	
	\begin{prop}
		If $p = 1/n$ or $p = (n-1)/n$, then the parity function is the only locally stable $p$-function on the hypercube, up to isomorphisms.
	\end{prop}

	\begin{proof}		
		We prove only for $p = (n-1)/n$; the proof for $p=1/n$ is similar. 			
		
		We will show that $f$ depends only on a single coordinate. Let $x$ be an initial point in the hypercube and $y$ its unique neighbor such that $f(x) \neq f(y)$. Denote the coordinate in which they differ by $i$. By local stability, every other neighbor $x'$ of $x$ has $f(x') = f(x)$, and every other neighbor $y'$ of $y$ has $f(y') = f(y)$.		
		
		Let $j\neq i$, let $\tilde{x}$ be the neighbor of $x$ that differs from $x$ in coordinate $j$, and let $\tilde{y}$ be the neighbor of $y$ that differs $y$ in coordinate $j$. Then $\tilde{x}$ is a neighbor of $\tilde{y}$, since $\tilde{x}$ and $\tilde{y}$ differ only in the $i$-th coordinate. Also, since $f(x) = f(\tilde{x})$ and $f(y) = f(\tilde{y})$ but $f(x) \neq f(y)$, we have $f(\tilde{x}) \neq f(\tilde{y})$. 
		
		Since $f$ is locally $(n-1)/n$-stable, each of $x$'s neighbors $x'$ has exactly one neighbor $y'$ on which $f$ attains the opposite value. By the above, for each such $x'$, the corresponding $y'$ differs from it in the $i$-th coordinate. This reasoning can be repeated, choosing a neighbor of $x$ as the initial starting point, showing that for all $x'$ with the same $i$-th coordinate as $x$, $f(x) = f(x')$, while for all $x'$ that differ in the $i$-th coordinate from $x$, $f(x) \neq f(x')$. This means that either $f(x) = x_i$ or $f(x) = -x_i$.
	\end{proof}
	
	\begin{figure}[h]
		\centering
		\includegraphics[scale = 0.5]{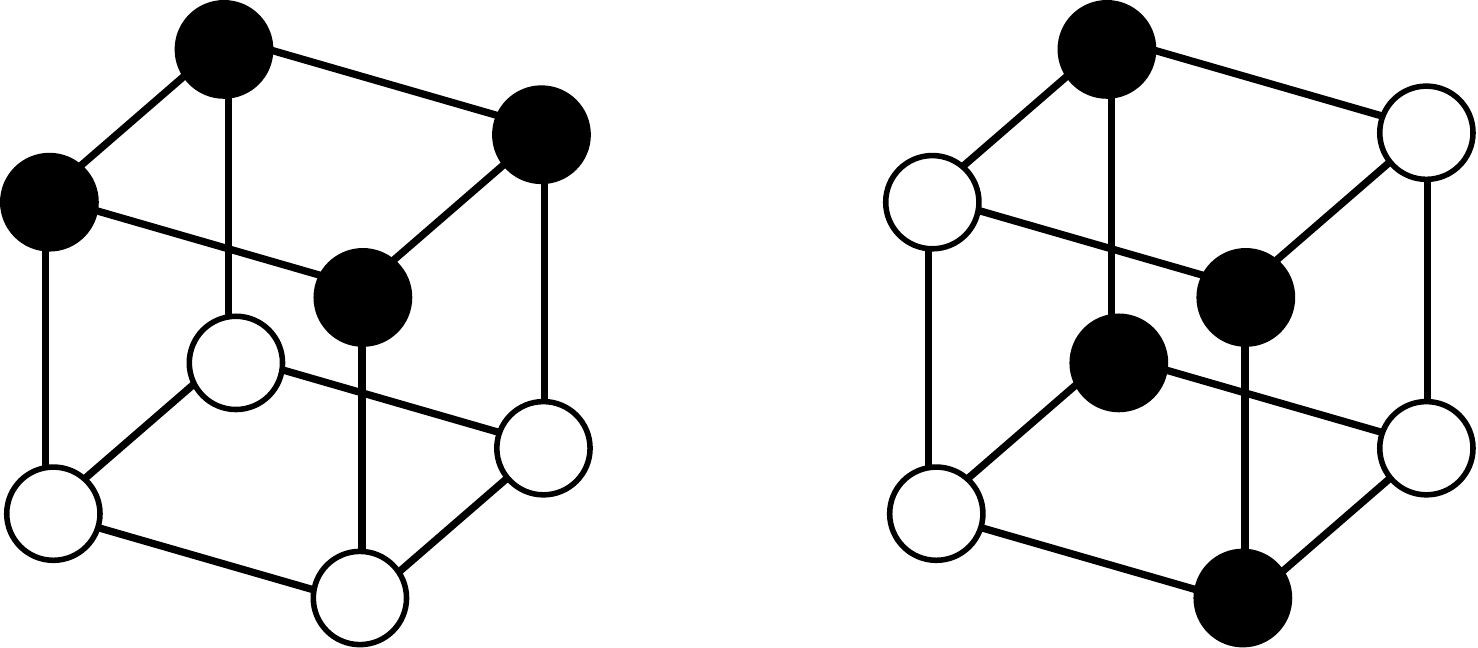}
		\caption{Left: The only locally $(n-1)/n$-stable function is the parity function on $1$ variable. Right: The only locally $1/n$-stable function is the parity function on $n-1$ variables.}
		\label{fig:parity}
	\end{figure}
	
	Many other $p$ values, however, have larger classes of non-isomorphic locally $p$-stable functions, since locally stable functions can be built out of locally $1/2$-biased functions:
	
	\begin{prop} \label{prop:locally_biased_to_locally_stable}
		Let $n > 0$ be an even integer. For every locally $1/2$-biased function $f$ on the $n$-dimensional hypercube, there exists a locally $(n/2)/(n+1)$-stable function $f'$ on the $n+1$-dimensional hypercube. Further, if $f$ and $g$ are two non-isomorphic locally $1/2$-biased functions, then $f'$ and $g'$ are also non-isomorphic. 		
	\end{prop}

	\begin{proof}
		Define $f'$ by		
		$$ f'(x_1,\ldots,x_{n+1}) = f(x_1,\ldots, x_n) \cdot x_{n+1}.$$		
		Let $x \in \set{-1,1}^{n+1}$ be a point in the $n+1$-dimensional hypercube. Since $f$ is locally $1/2$-biased, the function $f'$ attains the value $1$ on exactly half of $x$'s neighbors which differ from $x$ in one of the first $n$ coordinates, and the value $-1$ on the other half of these neighbors. Alternatively, $f'$ retains its value on exactly half of the neighbors which differ in the first $n$ coordinates. For the neighbor that is different from $x$ in the last coordinate, though, $f'$ flips its sign. Therefore $f'$ retains its value on a $(n/2)/(n+1)$ fraction of $x$'s neighbors, so $f$ is locally $(n/2)/(n+1)$-stable.
		
		The claim about non-isomorphism follows directly from the functions' Fourier decomposition.		
	\end{proof}

	Observe that unlike locally biased functions, locally stable functions can be easily extended to higher dimension:
	
	\begin{observation} \label{obv:n_m_stable}
		Let $f$ be a locally $(n-m)/n$-stable function. Then $f$ can be extended to hypercubes of size $n' \geq n$ by simply ignoring all but the first $n$ coordinates. This gives a locally $(n'-m)/n'$-stable function.
	\end{observation}

	We can use this observation to give a lower bound on the number of locally $(n'-m)/n'$-stable functions for a fixed $m$ and any $n' \geq 2m-2$. 
	This works as follows: first, pick any fixed $m>1$. Using Proposition \ref{prop:locally_biased_to_locally_stable}, we obtain a locally $(n-m)/n = (n/2)/(n+1)$-stable with $n = 2m - 2$. This can be extended by Observation \ref{obv:n_m_stable} to any $n'\geq n$, and together with Proposition \ref{prop:number_of_half_biased} we get a lower bound of $C2^{\sqrt{2m-2}}/(2m-2)^{1/4}$ different locally $(n'-m)/n'$-stable functions. 
	
	This observation also provides us with a pair of non-isomorphic locally stable functions for all hypercubes of dimension $n\geq 5$, showing that:
	\begin{cor}
		Scenery reconstruction is impossible for $n$-dimensional hypercubes for $n\geq5$.
	\end{cor}

	
	\section{Other directions and open questions}\label{sec:open_questions}
	In this section we discuss similar results and questions for other graphs. We also list some further questions regarding locally biased and locally stable functions on the hypercube.	For other excellent open problems see~\cite{finucane+tamuz+yaari2014}.
	
	\subsection{Hypercube reconstruction}
	Our work shows that in general, Boolean functions on the hypercube cannot be reconstructed.
	
	\begin{que}
		Under which conditions is it possible to reconstruct Boolean functions on the hypercube? 
	\end{que}

	\begin{que}
		Is a random Boolean function reconstructible with high probability?
	\end{que}

	\begin{remark}
		Using the techniques of \cite{benjamini+kesten96}, it can be shown that reconstruction is always possible in the hypercube of dimension at most $3$.
	\end{remark}
	
	\subsection{Other graphs}
	
	Note that the necessity condition on $p$ of Theorem \ref{thm:main_theorem} can be applied to any finite regular graph, ruling out functions based on the relation between the graph degree and the number of vertices. 
	
	\subsection*{Trees}
	Let $G$ be an $n$-regular infinite tree. Then for any $p = b/n$, $b = 0,1,\ldots,n$ there exists a locally $p$-biased function. Such a function can be found greedily by picking a root vertex $v \in G$, setting $f(v) = 1$, and iteratively assigning values to vertices further away in any way that meets the constraints. 
		
	Notice that the method above requires picking some initial vertex, and that the method yields many possible functions on labeled trees (all of which are isomorphic when we remove the labels). Once the initial vertex $v$ has been fixed, it is possible to generate a distribution on locally $p$-biased functions, by setting	$f(v)$ to be $1$ with probability $b/n$, and randomly expanding from there.
			
	\begin{que}
		For an $n$-regular tree $G$, find an invariant probability measure on locally $p$-biased functions that commutes with the automorphisms of the tree. 
	\end{que}	

	\subsection*{The standard lattice}
	
		\begin{figure}[h]
		\centering
		\includegraphics[scale = 0.5]{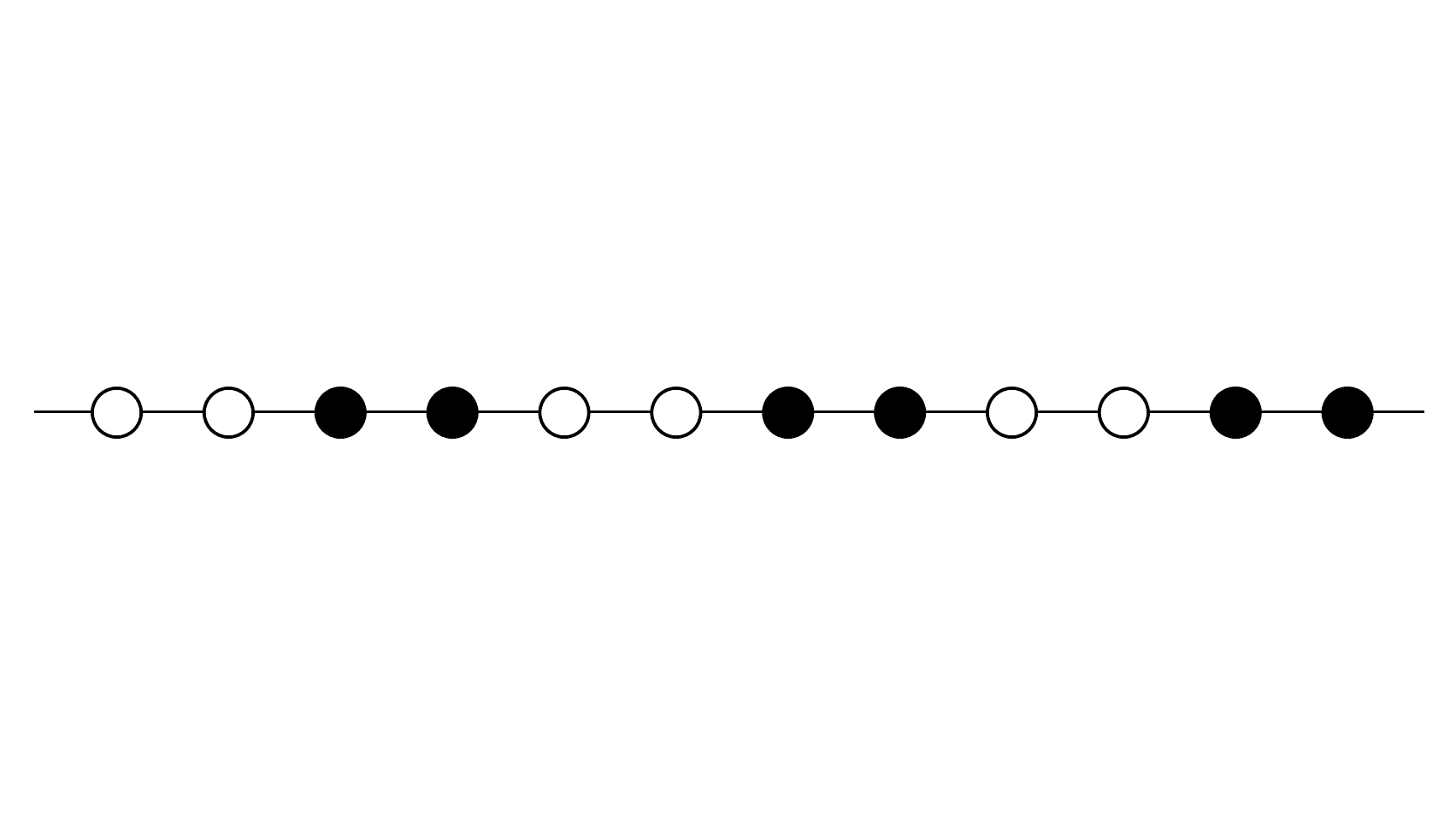}
		\includegraphics[scale = 0.5]{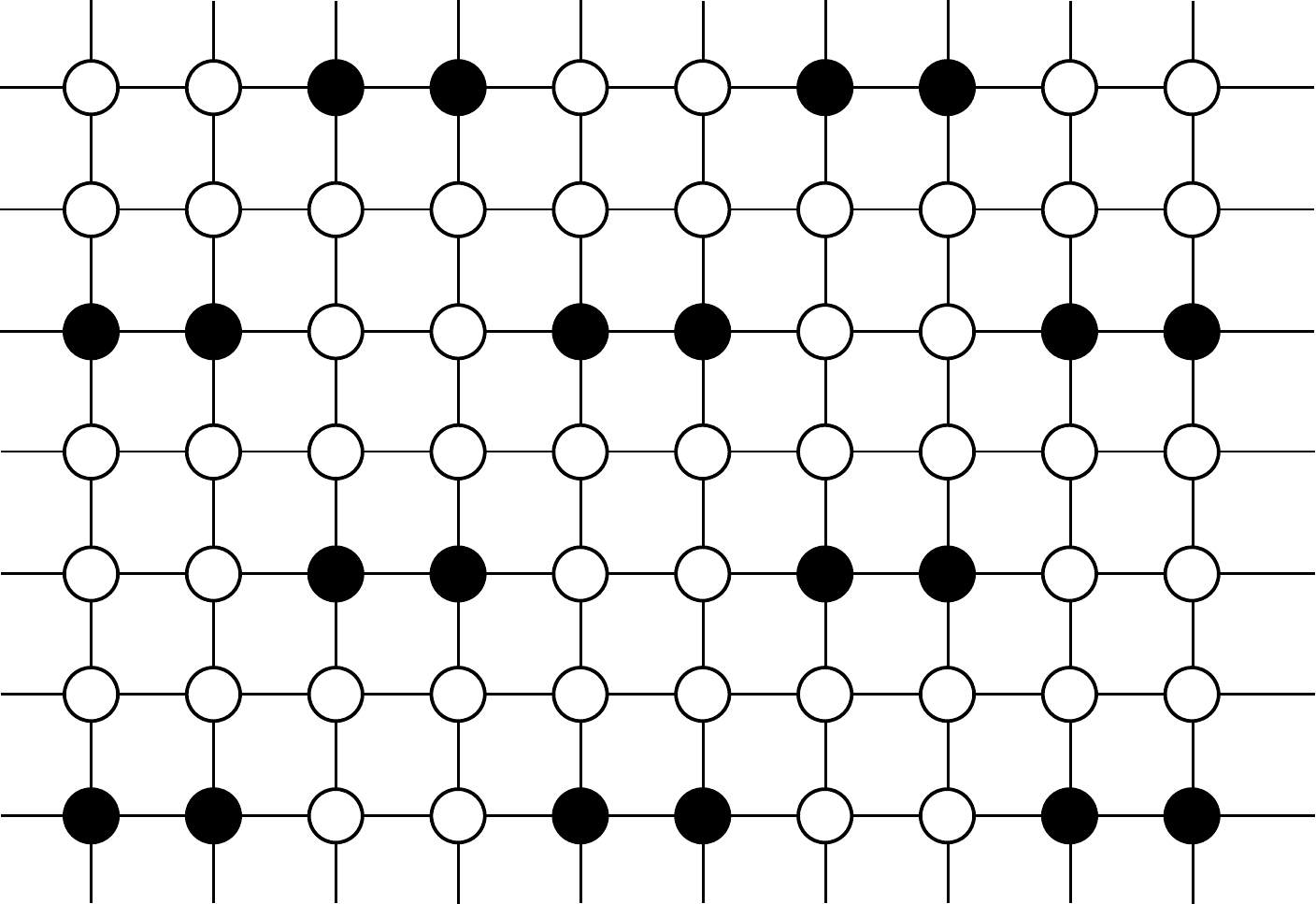}
		\caption{Left: a locally $1/2$-biased function on $\bZ$. Right: a locally $1/4$-biased function on $\bZ^2$. }
		\label{fig:z_locally_biased}
	\end{figure}	
	
	The following propositions show that there is a one-to-one mapping of locally $p$-biased functions from the hypercube to $\bZ^n$. Since automorphisms of the lattice can be pulled back to automorphisms of the hypercube, we get lower bounds for the size of non-isomorphic locally $p$-biased functions on $\bZ^n$.
	
	\begin{prop}\label{prop:embedding_to_zn}
		Let $f:\{-1,1\}^n\to \{-1,1\}$ be a locally $p$-biased function. 
		Then there exists a locally $p$-biased extension $\tilde{f}:\bZ^d\to \{-1,1\}$ such that $\big.f\big|_{\{-1,1\}^n}=f$. In addition, if $f$ and $g$ are non-isomorphic locally $p$-biased functions on the hypercube, then $\tilde{f}$ and $\tilde{g}$ are non-isomorphic.
	\end{prop}
	
	\begin{proof}
		Here we think of the hypercube as $\{0, 1\}^n$ instead of $\{-1,1\}^n$.
		Let 
		$$\psi(x_1,\ldots,x_n)=(x_1\textrm{ mod }2,\ldots,x_n\textrm{ mod }2).$$ 
		For any $t=(t_1,\ldots,t_n)\in\bZ^n$ define $Q_t=t+\{0,1\}^n$. Then the set $\psi(Q_t)$ is the hypercube $\{0, 1\}^n$. Moreover, there exists an automorphism $\varphi$ of the hypercube, such that $\varphi(\psi(Q_t))=Q_t-t$.
		We define
		$$\widetilde{f}(x)=f(\psi(x)).$$
		Suppose that $f$ is locally $p$-biased function on the hypercube. Note that $\big.\widetilde{f}\big|_{Q_t}$ is locally $p$-biased on $Q_t$ for any $t\in\bZ^n$. 
		Let $x\in\bZ^n$. Consider $Q^+=Q_x$ and $Q^-=Q_{x-e}$, where $e=(1,\ldots,1)$.
		Then, $Q^+\cap Q^- = \{x\}$ and the neighbors of $x$ are partitioned such that half of them are in $Q^-$ and the other half are in $Q^+$.
		Since $\big.\widetilde{f}\big|_{Q_x^{\pm}}$ is locally $p$-biased, out of the $n$ neighbors of $x$ in $Q^{+}$ on exactly $pn$ the value of $\widetilde{f}$ is $1$, and the same holds true for $Q^-$. Thus, $\widetilde{f}$ is locally $p$-biased function on $\bZ^n$.\\
		Note that the automorphisms of $\bZ^n$ are those of the hypercube with the addition of translations.
		Since there exists an isomorphism between any two hypercubes in the tiling (the above mentioned $\varphi$), any automorphism between $\widetilde{f}$ and $\widetilde{g}$ would induce one between $f$ and $g$.
	\end{proof}
	
	The above extension procedure gives us lower bounds on the growth rate of some classes of non-isomorphic locally $p$-biased functions.
	
	\begin{cor}
		Let $\widetilde{B}_{p}^n$ be the class of non-isomorphic locally $p$-biased functions on $\bZ^n$. 
		\begin{enumerate}
			\item 
			If $n$ is even, then $\left|\widetilde{B}_{1/2}^n\right| \geq C2^{\sqrt{n}}/n^{1/4}$, where $C>0$ is a universal constant.
			\item
			If $n=2^{-m}$, then $\abs{\widetilde{B}_{1/n}^n}$ is super-exponential.
		\end{enumerate}
	\end{cor}

	Unlike for the hypercube, we do not have a characterization theorem for the lattice $\bZ^n$. In fact, we have found a locally $1/2$-biased function for $\bZ$ and a locally $1/4$-biased function for $\bZ^2$; see Figure \ref{fig:z_locally_biased}. Both of these are not the result of embedding the relevant hypercube in the lattice via Proposition \ref{prop:embedding_to_zn}. 
	
	\begin{que}
	Give a complete characterization of permissible $p$ values for locally $p$-biased functions on $\bZ^n$. When such functions exist, count how many there are.
	\end{que}

	\subsection*{Cayley Graphs}
	In general, for a given group with a natural generating set, it is interesting to ask whether its Cayley graph admits locally biased or locally stable functions, and if so, how many. Specific examples which spring to mind for such groups are the group of permutations $S_n$ with either all transpositions $\set{\sigma_{ij}}_{i<j}$, and $\bZ$ with any number of generators. For the latter case, the following observation shows that for any two generators, $\bZ$ has a locally $1/2$-biased function:
	
	\begin{observation}	
	Let $a > 1$ and $b > 1$ generate $\bZ$. Then the function $f$ defined by 
		$$ f(x) = \begin{cases}
		
		1, & 0 \leq (x \mod 2(a+b)) < a+b  \\
		-1, & a+b \leq (x \mod 2(a+b)) < 2(a+b)  \\
		\end{cases} $$	
	is locally $1/2$-biased.
	\end{observation}
	
	Computer search shows that for some generators, other locally biased functions exist; see Figure \ref{fig:z_cayley_example} for an example.

	\begin{que}
		Characterize the locally biased and locally stable functions on $S_n$ as a function of its generating set.
	\end{que}
	
	\begin{que}
		Characterize the locally biased and locally stable functions on $\bZ$ as a function of its generating set.
	\end{que}

	\begin{figure}[h]
		\centering
		\includegraphics[scale = 0.5]{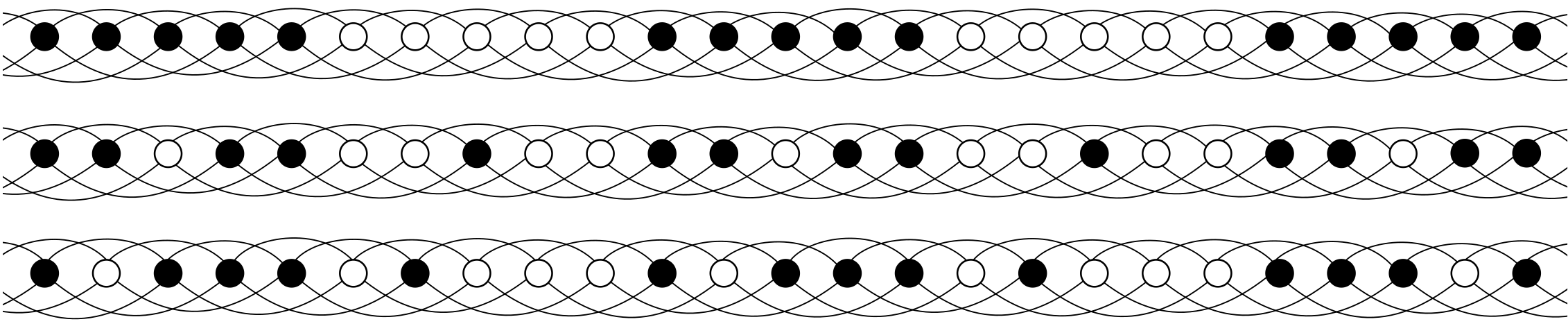}
		\caption{Three non-isomorphic locally $1/2$-biased functions for $\bZ$ with the generators $\set{2,3}$. Computer search shows that these are the only ones.}
		\label{fig:z_cayley_example}
	\end{figure}	

	\subsection{Locally biased and locally stable functions}
	Section \S~\ref{sec:non_unique} only gives lower bounds on the number of locally biased functions, and applies only for $p=1/2$ and $p = 1/n$ (and $1-1/n$ by taking negation of functions).
	
	\begin{que}
		What are the exact asymptotics for the number of non-isomorphic locally biased functions, for all permissible $p$?
	\end{que}
	
	We can also ask about the robustness of the locally biased property:
	
	\begin{que}
		How do the characterization and counting theorems for locally biased functions change, when we relax the locally biased demand for $2^{o(n)}$ of the vertices (i.e a small amount of vertices can have their neighbors labeled arbitrarily)?
	\end{que}

	The uniqueness of locally $1/n$-stable functions is in stark contrast to the exponential size of locally $1/n$-biased functions. Our bounds in section \S~\ref{sec:locally_stable_functions} for the number of $(n-m)/n$-locally stable functions are exponential in $m$, but not in $n$. We seek a better understanding of these functions:
	
	\begin{que}
		What are the exact asymptotics for the number of non-isomorphic locally stable functions?
	\end{que}

	
	\section{Acknowledgments}
	We thank Itai Benjamini for proposing the question of indistinguishability and for his advice, Ronen Eldan for his suggestions on locally stable functions, and David Ellis for the connection to perfect codes. We also thank Noga Alon and Peleg Michaeli for some useful discussions.
	
	\bibliography{locally_biased}
	\bibliographystyle{plain}

\end{document}